\documentclass[11pt, oneside]{amsart}
\usepackage{amsmath,amsthm,amscd,amssymb,latexsym,upref}
\date{\today}

\newcommand{\bbZ}{{\mathbb{Z}}}
\newcommand{\bbC}{{\mathbb{C}}}

\newcommand{\bbT}{{\mathbb{T}}}

\newcommand{\eqbg}{\begin{equation}\label}
\newcommand{\eqen}{\end{equation}}

\allowdisplaybreaks
\numberwithin{equation}{section}
\newtheorem{theorem}{Theorem}[section]

\newtheorem{lemma}[theorem]{Lemma}

\newtheorem{corollary}[theorem]{Corollary}

\theoremstyle{definition}

\newtheorem{remark}[theorem]{Remark}

\begin{document}

\title[]{On Scattering for CMV Matrices}
\author[]{ A. Kheifets, F. Peherstorfer,
and P. Yuditskii}

\thanks{The work was partially supported by
the Austrian Science Found FWF, project number: P16390--N04 }

\address{
Department of Mathematics, University of Massachusetts Lowell,
Lowell, MA, 01854, USA} \email{Alexander\_Kheifets@uml.edu}

\address{Institute for Analysis, Johannes Kepler University Linz,
A-4040 Linz, Austria}

\email{franz.peherstorfer@jku.at}

\email{petro.yuditskiy@jku.at}

\date{\today}

\maketitle

\begin{abstract}
Adamjan-Arov \cite{AA} (Lax--Phillips \cite{LP}) model space is
considered as a scattering representation space for a CMV matrix
\cite{Sv1} in context of new (extended Marchenko--Faddeev
\cite{Mar}) scattering theory developed in \cite{PeVoYu, TTh, VYu}.
That is, there exists a basis in which the multiplication by
independent variable is a CMV matrix. This basis as well as
Verblunski coefficients are computed explicitly in terms of Nehari
interpolation \cite{AAK, AK}. Asymptotically the Verblynski
coefficients go to zero. Moreover, relations between the basis and
wandering subspaces are established. Transformation from scattering
representation to spectral representation is given.
\end{abstract}

\section{Space $L^R$, its subspaces and functional models}

\bigskip\noindent Let $R(t)$ be a given Szeg\"o contractive function $R$
on the unite circle $\bbT$
  \begin{equation}\label{Szege-0}
  |R(t)|\le 1,\quad \log(1-|R|)\in L^1.
  \end{equation}
Then there exists a unique outer function $T$ such that
$$
|T|^2=1-|R|^2,\ \  T(0)>0.
$$
We consider the space $L^R$ of vector functions on $\bbT$ with the
following weighted inner product on it
\begin{equation}
\left < \begin{bmatrix} 1 & \overline R \\
R & 1 \end{bmatrix}^{-1} \begin{bmatrix} f_1 \\ f_2 \end{bmatrix},
\begin{bmatrix} f_1 \\ f_2 \end{bmatrix}\right >=
\left < \frac{1}{|T|^2}\begin{bmatrix} 1 & -\overline R \\
-R & 1 \end{bmatrix} \begin{bmatrix} f_1 \\ f_2
\end{bmatrix},
\begin{bmatrix} f_1 \\ f_2 \end{bmatrix}\right >.
\label{inner-prod-01}
\end{equation}
Since the weight matrix is greater than $\begin{bmatrix} 1 & 0 \\
0 & 0 \end{bmatrix}$ and greater than $\begin{bmatrix} 0 & 0 \\
0 & 1 \end{bmatrix}$, the entries $f_1$ and $f_2$, in particular,
belong to $L^2$. The set of functions
\begin{equation}\nonumber
\begin{bmatrix} 1 & \overline R \\
R & 1 \end{bmatrix} \begin{bmatrix} g_1 \\ g_2
\end{bmatrix}
\end{equation}
($g_1\in L^2$ and $g_2\in L^2$) is dense in $L^R$. We denote by
$U_R$ the (unitary) operator of multiplication by independent
variable $t$ on $L_R$. We consider the following subspaces in
$L^R$
\begin{equation}
\check{\mathcal H}_{n,m} = {\rm Clos}\left\{
\begin{bmatrix} 1 & \overline R \\
R & 1 \end{bmatrix}
\begin{bmatrix} t^n H^2_+ \\
{\overline t}^{\ m} {H^2_-}\end{bmatrix}\right\}, \label{H-nm-01}
\end{equation}
where $n$ and $m$ are integers. The subspaces appear in the
context of the Nehari problem if the Fourier coefficients of $R$
with indices from $-\infty$ to $-(n+m+1)$ are specified. Note that
$$
\check{\mathcal H}_{n,m}\supseteq \check{\mathcal
H}_{n,m+1}\quad{\rm and}\quad \check{\mathcal H}_{n,m}\supseteq
\check{\mathcal H}_{n+1,m},
$$
and that due to the Szeg\"o property (\ref{Szege-0}) of $R$,
$\check{\mathcal H}_{n,m+1}$ and $\check{\mathcal H}_{n+1,m}$ are
of codimension one in $\check{\mathcal H}_{n,m}$.
 Define
      subspaces ${\check\Delta}_{n,m}$ and $\widetilde{\check\Delta}_{n,m}$ of
      $\check{\mathcal H}_{n,m}$ by
      \begin{equation}  \label{defectspaces}
      {\check\Delta}_{n,m} = \check{\mathcal H}_{n,m} \ominus \check{\mathcal H}_{n+1,m}, \qquad
       \widetilde{\check\Delta}_{n,m} = \check{\mathcal H}_{n,m} \ominus \check{\mathcal H}_{n,m+1}.
      \end{equation}
$U_R$ maps $\check{\mathcal H}_{n,m+1}$ onto $\check{\mathcal
H}_{n+1,m}$ and, therefore, $\check{\mathcal
H}_{n,m}^\perp\oplus\widetilde{\check\Delta}_{n,m}$ onto
$\check{\mathcal H}_{n,m}^\perp\oplus{\check\Delta}_{n,m}$.

Let $\Delta=\widetilde{\Delta}=\bbC$ with unitary identification
maps
      $$ {\check i}_{n,m} \colon \Delta \mapsto
   {\check\Delta}_{n,m}, \qquad
    \widetilde {\check i}_{n,m} \colon \widetilde\Delta \mapsto
    \widetilde{\check\Delta}_{n,m}.
      $$
They are defined up to multiplying by unitary constants and we
will choose specific normalizations later. We denote by ${\check
K}_{n,m}$ and $\widetilde {\check K}_{n,m}$ the images of $1$
under ${\check i}_{n,m}$ and $\widetilde {\check i}_{n,m}$,
respectively and we will get "explicit" formulas for the vectors
${\check K}_{n,m}$ and $\widetilde {\check K}_{n,m}$ using a
construction motivated by the solution of the Nehari problem.
      Define a unitary operator
$$
       {\check U}_{0,n,m}  = \begin{bmatrix} A_{0,n,m} & B_{0,n,m} \\
        C_{0,n,m} & 0 \end{bmatrix}
        \colon \begin{bmatrix} \check{\mathcal
        H}_{n,m} \\ \Delta\end{bmatrix}
        \mapsto
       \begin{bmatrix} \check{\mathcal H}_{n,m} \\
       \widetilde \Delta \end{bmatrix},
      $$
using the decompositios
      $$
        {\check U}_{0,n,m} \colon  \begin{bmatrix} {\check{\mathcal H}_{n,m+1}}
        \oplus
        \widetilde {\check\Delta}_{n,m} \\ \Delta
        \end{bmatrix} \mapsto
        \begin{bmatrix} {\check{\mathcal H}_{n+1,m}} \oplus
        {\check\Delta}_{n,m} \\ \widetilde \Delta
        \end{bmatrix}
      $$
by
      \begin{equation}  \label{U0-1}
      {\check U}_{0,n,m}\colon \begin{bmatrix} h_{n,m+1} +\widetilde  {\check K}_{n,m}\cdot\beta  \\
      \alpha\end{bmatrix} \mapsto
       \begin{bmatrix}
      U_R h_{n,m+1} + {\check K}_{n,m}\cdot\alpha
       \\ \beta
       \end{bmatrix}
      \end{equation}
and a unitary operator $ {\check U}_{1,n,m}:\check{\mathcal
H}_{n,m}^\perp\oplus\widetilde\Delta \to \check{\mathcal
H}_{n,m}^\perp\oplus\Delta $ by
$$
{\check U}_{1,n,m}=\begin{bmatrix} A_{1,n,m} & B_{1,n,m} \\
        C_{1,n,m} & D_{1,n,m} \end{bmatrix}
        : \begin{bmatrix} h_{n,m}^\perp  \\
      \beta\end{bmatrix} \mapsto
      \begin{bmatrix} (h_{n,m}^{\perp})'  \\
      \alpha\end{bmatrix},
$$
such that
$$U_R: h_{n,m}^\perp +
      \widetilde {\check K}_{n,m}\cdot\beta\mapsto
      (h_{n,m}^{\perp})' + {\check K}_{n,m}\cdot
      \alpha
$$
Then $U_R$ is the feedback coupling of ${\check U}_{0,n,m}$ with
${\check U}_{1,n,m}$ given by the formula
\begin{equation}\label{U}
        U_R = 
        \begin{bmatrix} A_{0,n,m} +
        B_{0,n,m}D_{1,n,m}C_{0,n,m} & B_{0,n,m}C_{1,n,m} \\ B_{1,n,m} C_{0,n,m} & A_{1,n,m}
        \end{bmatrix}.
      \end{equation}
In addition to ${\check U}_{0,n,m}$ we consider it's unitary
dilation $\mathcal {\check U}_{0,n,m}$ that acts on
$$
\mathcal L_{n,m}=...\Delta\oplus\Delta\oplus \check{\mathcal
H}_{n,m}\oplus\widetilde\Delta\oplus\widetilde\Delta ...
$$
Subspaces of $L_R$
$$
\begin{bmatrix} 1 \\ R \end{bmatrix}L^2\quad {\rm and}\quad
\begin{bmatrix} \overline R \\ 1 \end{bmatrix}L^2
$$
can be embedded isometrically into $\mathcal L_{n,m}$ as follows.
Since $$\begin{bmatrix} 1 \\ R \end{bmatrix}t^n H^2_+\subset
\check{\mathcal H}_{n,m}\subset\mathcal L_{n,m}$$ and $$\begin{bmatrix} \overline R \\
1 \end{bmatrix}\overline t^m H^2_-\subset \check{\mathcal
H}_{n,m}\subset\mathcal L_{n,m}\ ,$$ we can define the embedding
as
$$
{\check i}_{n,m}':\begin{bmatrix} 1 \\ R \end{bmatrix}t^j H^2_+\to
\mathcal {\check U}_{0,n,m}^{j-n}\begin{bmatrix} 1 \\ R
\end{bmatrix}t^n H^2_+ ,\quad j\in\bbZ
$$
and
$$
{\check i}_{n,m}'':\begin{bmatrix} \overline R \\ 1
\end{bmatrix}\overline t^j H^2_-\to \mathcal {\check U}_{0,n,m}^{m-j}\begin{bmatrix}
\overline R \\ 1
\end{bmatrix}\overline t^m H^2_- ,\quad j\in\bbZ .
$$
They are well defined isometric mappings from dense linear subsets
of $\begin{bmatrix} 1 \\ R \end{bmatrix}L^2$ and $\begin{bmatrix}
\overline R \\ 1 \end{bmatrix}L^2$, respectively. Therefore, they
extend by continuity.

We consider four scales
\begin{eqnarray}\nonumber
&&\Psi_{0,n,m}: \Delta\to {\check i}_{n,m}\Delta,\quad
\widetilde\Psi_{0,n,m}: \widetilde\Delta\to \mathcal {\check
U}_{0,n,m}\widetilde
{\check i}_{n,m}\widetilde\Delta,\\
&&\Psi_{0,n,m}': E'\ \to {\check i}_{n,m}'\begin{bmatrix} 1 \\ R
\end{bmatrix}E'=\mathcal {\check U}_{0,n,m}^{-n} \begin{bmatrix} 1 \\ R
\end{bmatrix}t^n E',\nonumber\\
&&\Psi_{0,n,m}'': E''\to {\check i}_{n,m}''\begin{bmatrix} \overline R \\
1 \end{bmatrix}E''=\mathcal {\check U}_{0,n,m}^{m+1}\begin{bmatrix} \overline R \\
1 \end{bmatrix}\overline t^{\ m+1} E''\nonumber.
\end{eqnarray}
We also define the characteristic (scattering) measure of
$\mathcal {\check U}_{0,n,m}$ with respect to the vector scale
\begin{align}
   &
    \begin{bmatrix} \Psi_{0,n,m}^{*} \\
     \widetilde \Psi_{0,n,m}^{*} \\ \Psi_{0,n,m}^{\prime *}
     \\  \Psi_{0,n,m}^{\prime \prime *}\end{bmatrix} E_{\mathcal {\check U}_{0,n,m}}
    \begin{bmatrix} \Psi_{0,n,m} &
     \widetilde \Psi_{0,n,m} & \Psi_{0,n,m}^{\prime }
     &  \Psi_{0,n,m}^{\prime \prime }\end{bmatrix},
      \label{U0charfunc-full-0}
  \end{align}
where $E_{\mathcal {\check U}_{0,n,m}}$ is the spectral measure of
the unitary operator $ {\check U}_{0,n,m}$. It is absolutely
continuous with respect to the Lebesgue measure. Its density can
be formally defined as (precisely it is the boundary value of a
harmonic function)
\begin{align}
   {\check\Sigma}_{0,n,m}: &=&
    \begin{bmatrix} \Psi_{0,n,m}^{*} \\
     \widetilde \Psi_{0,n,m}^{*} \\ \Psi_{0,n,m}^{\prime *}
     \\  \Psi_{0,n,m}^{\prime \prime *}\end{bmatrix}
     \sum\limits _{k=-\infty}^\infty
     t^k\mathcal {\check U}_{0,n,m}^{*\ k}
      \begin{bmatrix} \Psi_{0,n,m} &
     \widetilde \Psi_{0,n,m} & \Psi_{0,n,m}^{\prime }
     &  \Psi_{0,n,m}^{\prime \prime }\end{bmatrix}
     \nonumber\\
  &=& \begin{bmatrix} 1 & {\check b}_{n,m} &
      {\check s}_{1,n,m} & 0 \\
      \overline {\check b}_{n,m} & 1 & 0 &
      \overline {\check s}_{2,n,m} \\
      \overline {\check s}_{1,n,m} & 0 & 1 & \overline {\check s}_{0,n,m}\\
      0 & {\check s}_{2,n,m} & {\check s}_{0,n,m} &
      1
\end{bmatrix}:\begin{bmatrix} \Delta \\ \widetilde\Delta \\
E' \\ E'' \end{bmatrix}\to\begin{bmatrix} \Delta \\ \widetilde\Delta \\
E' \\ E'' \end{bmatrix}.
      \label{U0charfunc-full-density-0}
  \end{align}
Since the scale $\begin{bmatrix} \Psi_{0,n,m}^{\prime }
     &  \Psi_{0,n,m}^{\prime \prime }\end{bmatrix}$
     is $*$-cyclic for $\mathcal {\check U}_{0,n,m}$, we have
$$
\begin{bmatrix} 1 & {\check b}_{n,m}\\ \overline {\check b}_{n,m} & 1 \end{bmatrix}=
\begin{bmatrix} {\check s}_{1,n,m} & 0 \\ 0 & \overline {\check s}_{2,n,m} \end{bmatrix}
\begin{bmatrix} 1 & \overline {\check s}_{0,n,m}\\ {\check s}_{0,n,m} & 1
\end{bmatrix}^{-1}
\begin{bmatrix} \overline {\check s}_{1,n,m} & 0\\ 0 & {\check s}_{2,n,m}
\end{bmatrix}.
$$
Therefore, the matrix
$$
{\check S}_{n,m}
   = \begin{bmatrix} {\check b}_{n,m} &
      {\check s}_{1,n,m} \\
      {\check s}_{2,n,m} & {\check s}_{0,n,m}
\end{bmatrix}:\begin{bmatrix}\widetilde\Delta \\
E' \end{bmatrix}\to\begin{bmatrix} \Delta \\ E'' \end{bmatrix}.
$$
is unitary almost everywhere on $\bbT$.

The entries of ${\check S}_{n,m}$ have these properties: ${\check
b}_{n,m}$ is analytic and ${\check b}_{n,m}(0)=0$ (although it is
not important to us, note that ${\check b}_{n,m}$ is a
characteristic function of the colligation ${\check
U}_{0,n,m}^*$); ${\check s}_{1,n,m}=\overline t^{\ n}{\check
a}_{n,m}$ and ${\check s}_{2,n,m}=\overline t^{\ m}{\check
a}_{n,m}$, where ${\check a}_{n,m}$ is an outer function. We can
fix now normalization of ${\check i}_{n,m}$ and $\widetilde
{\check i}_{n,m}$ (in a unique way) such that ${\check a}_{n,m}
(0)>0$. Note that if $n+m=n'+m'$ then
$$
U_R^{n'-n} : \check{\mathcal H}_{n,m} \to  \check{\mathcal
H}_{n',m'}
$$
Therefore, unitary colligations ${\check U}_{0,n,m}$ and ${\check
U}_{0,n',m'}$ are unitarily equivalent. Hence, (under the above
normalization of ${\check i}_{n,m}$ and $\widetilde {\check
i}_{n,m}$) we have ${\check b}_{n,m}={\check b}_{n',m'}$,
consequently, ${\check a}_{n,m}={\check a}_{n',m'}$. We also get
$$
{\check s}_{0,n,m}=-\frac{{\check s}_{2,n,m}}{\overline {\check
s}_{1,n,m}} {\check b}_{n,m}= -\overline t^{\ m+n}\frac{{\check
a}_{n,m}}{\overline {\check a}_{n,m}} {\check b}_{n,m}= -\overline
t^{\ m'+n'}\frac{{\check a}_{n',m'}}{\overline {\check a}_{n',m'}}
{\check b}_{n',m'}= {\check s}_{0,n',m'}
$$
Let $\check\omega_{n,m}$ be the characteristic function of
${\check U}_{1,n,m}^*,\
\check\omega_{n,m}(\zeta):\Delta\to\widetilde\Delta$. Analogous to
the above is that $\check\omega _{n,m}=\check\omega_{n',m'}$ if
$n+m=n'+m'$. In view of the above remarks sometimes we will use
the notations ${\check b}_{n+m}, {\check a}_{n+m}, {\check
s}_{0,n+m}, \check\omega _{n+m}$ instead of the former ones. We
can also write ${\check s}_{1,n,m}=\overline t^{\ n}{\check
a}_{n+m}$ and ${\check s}_{2,n,m}=\overline t^{\ m}{\check
a}_{n+m}$. Since $U_R$ is the feedback coupling of ${\check
U}_{0,n,m}$ with ${\check U}_{1,n,m}$ we get that
$$
R={\check s}_{0,n+m}+\overline t^{\ n+m} \frac{{\check
a}_{n+m}^2\check\omega_{n+m}}{1-\check\omega_{n+m}{\check
b}_{n+m}}= \overline t^{\ n+m} \cdot\frac{{\check
a}_{n+m}}{\overline {\check a}_{n+m}}\cdot\frac{\check\omega
_{n+m}-\overline {\check b}_{n+m}}{1-\check\omega_{n+m}{\check
b}_{n+m}}\ .
$$
To formulate one more crucial property of ${\check\Sigma}_{0,n,m}$
we need a functional model of $\mathcal {\check U}_{0,n,m}$. The
Fourier representation is defined as
$$
\begin{bmatrix} \Psi_{0,n,m}^{*} \\
     \widetilde \Psi_{0,n,m}^{*} \\ \Psi_{0,n,m}^{\prime *}
     \\  \Psi_{0,n,m}^{\prime \prime *}\end{bmatrix}
     \sum\limits _{k=-\infty}^\infty
     t^k\mathcal {\check U}_{0,n,m}^{*\ k}: \mathcal L_{n,m}\to
     L^{{\check\Sigma}_{0,n,m}}.
     $$
     Since ${\check U}_{0,n,m}$ goes to multiplication by $t$ under this
     transform and due to (\ref{U0charfunc-full-0}) we get that
     $$
     \begin{bmatrix} 1 \\ R \end{bmatrix}t^n H^2_+\to
     \begin{bmatrix} 1 & {\check b}_{n+m} &
      {\check s}_{1,n,m} & 0 \\
      \overline {\check b}_{n+m} & 1 & 0 &
      \overline {\check s}_{2,n,m} \\
      \overline {\check s}_{1,n,m} & 0 & 1 & \overline {\check s}_{0,n+m}\\
      0 & {\check s}_{2,n,m} & {\check s}_{0,n+m} &
      1
\end{bmatrix}
     \begin{bmatrix}
      0 \\ 0 \\ t^n H^2_+ \\ 0
 \end{bmatrix}
     $$
     and
     $$
\begin{bmatrix} \overline R \\ 1 \end{bmatrix}\overline t^{\ m}H^2_-\to
\begin{bmatrix} 1 & {\check b}_{n+m} &
      {\check s}_{1,n,m} & 0 \\
      \overline {\check b}_{n+m} & 1 & 0 &
      \overline {\check s}_{2,n,m} \\
      \overline {\check s}_{1,n,m} & 0 & 1 & \overline {\check s}_{0,n+m}\\
      0 & {\check s}_{2,n,m} & {\check s}_{0,n+m} &
      1
\end{bmatrix}
\begin{bmatrix}  0 \\ 0 \\ 0\\ \overline t^{\ m}H^2_-
\end{bmatrix}.
     $$
     Thus,
     \begin{equation}
     \check{\mathcal H}_{n,m}\to\ {\rm Clos \ }
     \begin{bmatrix} 1 & {\check b}_{n+m} &
      {\check s}_{1,n,m} & 0 \\
      \overline {\check b}_{n+m} & 1 & 0 &
      \overline {\check s}_{2,n,m} \\
      \overline {\check s}_{1,n,m} & 0 & 1 & \overline {\check s}_{0,n+m}\\
      0 & {\check s}_{2,n,m} & {\check s}_{0,n+m} &
      1
\end{bmatrix}
      \begin{bmatrix} 0\\0\\t^n H^2_+ \\ \overline t^{\ m}H^2_-
      \end{bmatrix}.\label{H-nm-Sigma-0}
     \end{equation}
     Also
     $$
     \ldots\Delta\oplus\Delta\to
\begin{bmatrix} 1 & {\check b}_{n+m} &
      {\check s}_{1,n,m} & 0 \\
      \overline {\check b}_{n+m} & 1 & 0 &
      \overline {\check s}_{2,n,m} \\
      \overline {\check s}_{1,n,m} & 0 & 1 & \overline {\check s}_{0,n+m}\\
      0 & {\check s}_{2,n,m} & {\check s}_{0,n+m} &
      1
\end{bmatrix}
\begin{bmatrix} H^2_- \\ 0 \\ 0 \\ 0 \end{bmatrix}
     $$
and
$$
     \widetilde\Delta\oplus\widetilde\Delta\ldots\to
\begin{bmatrix} 1 & {\check b}_{n+m} &
      {\check s}_{1,n,m} & 0 \\
      \overline {\check b}_{n+m} & 1 & 0 &
      \overline {\check s}_{2,n,m} \\
      \overline {\check s}_{1,n,m} & 0 & 1 & \overline {\check s}_{0,n+m}\\
      0 & {\check s}_{2,n,m} & {\check s}_{0,n+m} &
      1
\end{bmatrix}
\begin{bmatrix} 0\\ H^2_+ \\ 0 \\ 0 \end{bmatrix}.
     $$
     This implies that
     \begin{equation}\label{K-nm-0}
     {\check K}_{n,m}\to
     \begin{bmatrix} 1 \\
      \overline {\check b}_{n+m}  \\
      \overline {\check s}_{1,n,m} \\
      0
\end{bmatrix}={\check\Sigma}_{0,n,m}
\begin{bmatrix} 1\\0\\0\\0 \end{bmatrix}
\end{equation}
and
\begin{equation}\label{tilde-K-nm-0}
\widetilde {\check K}_{n,m}\to
\begin{bmatrix} {\check b}_{n+m} \\ 1  \\ 0 \\ {\check s}_{2,n,m}
\end{bmatrix}\overline t
={\check\Sigma}_{0,n,m}
\begin{bmatrix} 0\\ \overline t\\0\\0 \end{bmatrix}.
\end{equation}
Thus, we arrive at a crucial property of the entries of the matrix
${\check\Sigma}_{0,n,m}$:
\begin{equation}\label{approxim1}
\begin{bmatrix} 1 \\
      \overline {\check b}_{n+m}  \\
      \overline {\check s}_{1,n,m} \\
      0
\end{bmatrix}
\in \ {\rm Clos \ }
     \begin{bmatrix}
      {\check s}_{1,n,m} & 0 \\
      0 &
      \overline {\check s}_{2,n,m} \\
      1 & \overline {\check s}_{0,n+m}\\
      {\check s}_{0,n+m} &
      1
\end{bmatrix}
      \begin{bmatrix} t^n H^2_+ \\ \overline t^{\ m}H^2_-
      \end{bmatrix}
\end{equation}
and
\begin{equation}\label{approxim2}
\begin{bmatrix} {\check b}_{n+m} \\ 1  \\ 0 \\ {\check s}_{2,n,m}
\end{bmatrix}\overline t\in \ {\rm Clos \ }
     \begin{bmatrix}
      {\check s}_{1,n,m} & 0 \\
      0 &
      \overline {\check s}_{2,n,m} \\
      1 & \overline {\check s}_{0,n+m}\\
      {\check s}_{0,n+m} &
      1
\end{bmatrix}
      \begin{bmatrix} t^n H^2_+ \\ \overline t^{\ m}H^2_-
      \end{bmatrix}.
\end{equation}
In other notations
\begin{equation}\label{approxim1short}
{\check\Sigma}_{0,n,m}
\begin{bmatrix} 1\\0\\0\\0 \end{bmatrix}\in \ {\rm Clos \ }
{\check\Sigma}_{0,n,m}
\begin{bmatrix} 0\\0\\t^n H^2_+ \\ \overline t^{\ m}H^2_-
      \end{bmatrix},
\end{equation}
and
\begin{equation}\label{approxim2short}
{\check\Sigma}_{0,n,m}
\begin{bmatrix} 0\\ \overline t\\0\\0 \end{bmatrix}\in \ {\rm Clos \ }
{\check\Sigma}_{0,n,m}
\begin{bmatrix} 0\\0\\t^n H^2_+ \\ \overline t^{\ m}H^2_-
      \end{bmatrix}.
\end{equation}
Note that since the scale $\begin{bmatrix} \Psi_{0,n,m}^{\prime }
     &  \Psi_{0,n,m}^{\prime \prime }\end{bmatrix}$ is $*$-cyclic
$$
L^{{\check\Sigma}_{0,n,m}}=\ {\rm Clos \ }
     \begin{bmatrix}
      {\check s}_{1,n,m} & 0 \\
      0 &
      \overline {\check s}_{2,n,m} \\
      1 & \overline {\check s}_{0,n+m}\\
      {\check s}_{0,n+m} &
      1
\end{bmatrix}
      \begin{bmatrix} L^2 \\ L^2
      \end{bmatrix}
      =\ {\rm Clos \ }{\check\Sigma}_{0,n,m}
      \begin{bmatrix} 0\\0\\L^2 \\ L^2
      \end{bmatrix}
$$
and $L^{{\check s}_{0,n+m}}$ is mapped unitarily onto
$L^{{\check\Sigma}_{0,n,m}}$
 by
     $$ \begin{bmatrix} f' \\ f'' \end{bmatrix}\to
     \begin{bmatrix}
\begin{bmatrix} {\check s}_{1,n,m} & 0 \\ 0 & \overline {\check s}_{2,n,m} \end{bmatrix}
\begin{bmatrix} 1 & \overline {\check s}_{0,n+m}\\ {\check s}_{0,n+m} & 1
\end{bmatrix}^{-1}
\begin{bmatrix} f' \\ f'' \end{bmatrix}
\\ \\
\begin{bmatrix} f' \\ f'' \end{bmatrix}
\end{bmatrix}
     $$
However, this does not imply (\ref{approxim1short}) and
(\ref{approxim2short}).

A functional model for ${\check U}_{1,n,m}^*$ (using the scales
$\Psi_{n,m}$ and $\widetilde\Psi_{n,m}$ defined below) is the
"multiplication by $\overline t\ $" : ${\mathcal
H}^{\check\omega_{n+m}}\oplus\Delta\to {\mathcal
H}^{\check\omega_{n+m}}\oplus\widetilde\Delta$, where ${\mathcal
H}^{\check\omega_{n+m}}$ is a Hilbert space of vector functions
$$\begin{bmatrix}
h_-^{\check\omega_{n+m}}\\h_+^{\check\omega_{n+m}}\end{bmatrix},
\quad h_-^{\check\omega_{n+m}}\in H^2_-,\quad
h_+^{\check\omega_{n+m}}\in H^2_+
$$
with the norm
$$\Vert
\begin{bmatrix}
h_-^{\check\omega_{n+m}}\\h_+^{\check\omega_{n+m}}\end{bmatrix}\Vert^2=
\langle
\begin{bmatrix} 1 & \overline{\check\omega}_{n+m} \\ \check\omega_{n+m} & 1
\end{bmatrix}^{-1}
\begin{bmatrix}
h_-^{\check\omega_{n+m}}\\h_+^{\check\omega_{n+m}}\end{bmatrix},
\begin{bmatrix}
h_-^{\check\omega_{n+m}}\\h_+^{\check\omega_{n+m}}\end{bmatrix}
\rangle .
$$
${\check U}_{1,n,m}$ is realized as the "multiplication by $t$",
respectively.

We consider four scales associated to $U_R$
\begin{eqnarray}\nonumber
&&\Psi_{n,m}: \Delta\to {\check i}_{n,m}\Delta,\quad
\widetilde\Psi_{n,m}: \widetilde\Delta\to \mathcal U_{R}\widetilde
{\check i}_{n,m}\widetilde\Delta,\\
&&\Psi': E'\ \to \begin{bmatrix} 1 \\ R
\end{bmatrix}E',\quad\Psi'': E''\to \begin{bmatrix} \overline R \\
1 \end{bmatrix}E''\nonumber.
\end{eqnarray}
We consider the characteristic (scattering) measure of $\mathcal
U_{R}$ with respect to the vector scale
\begin{align}
   &
    \begin{bmatrix} \Psi_{n,m}^{*} \\
     \widetilde \Psi_{n,m}^{*} \\ \Psi^{\prime *}
     \\  \Psi^{\prime \prime *}\end{bmatrix} E_{\mathcal U_{R}}
    \begin{bmatrix} \Psi_{n,m} &
     \widetilde \Psi_{n,m} & \Psi^{\prime }
     &  \Psi^{\prime \prime }\end{bmatrix},
      \label{U-R-charfunc-full}
  \end{align}
where $E_{\mathcal U_{R}}$ is the spectral measure of the unitary
operator $ U_{R}$. It is absolutely continuous with respect to the
Lebesgue measure (this follows from $*$-cyclicity of the scale
$\begin{bmatrix} \Psi^{\prime }
     &  \Psi^{\prime \prime }\end{bmatrix}$).
Its density can be formally defined as (precisely it is the
boundary value of a harmonic function)
\begin{equation}
   {\check\Sigma}_{n,m}: =
    \begin{bmatrix} \Psi_{n,m}^{*} \\
     \widetilde \Psi_{n,m}^{*} \\ \Psi^{\prime *}
     \\  \Psi^{\prime \prime *}\end{bmatrix}
     \sum\limits _{k=-\infty}^\infty
     t^k\mathcal U_{R}^{*\ k}
      \begin{bmatrix} \Psi_{n,m} &
     \widetilde \Psi_{n,m} & \Psi^{\prime }
     &  \Psi^{\prime \prime }\end{bmatrix}.
     \nonumber
\end{equation}
Using feedback decomposition (\ref{U}) it can be expressed as
\begin{align}
   {\check\Sigma}_{n,m}
   = \begin{bmatrix} {\check\Sigma}_{n,m}^{(11)} &
   {\check\Sigma}_{n,m}^{(12)} \\ \\
   {\check\Sigma}_{n,m}^{(21)} &
   \Sigma^{(22)}
    \end{bmatrix}:\begin{bmatrix} \Delta \\ \widetilde\Delta \\
E' \\ E'' \end{bmatrix}\to\begin{bmatrix} \Delta \\ \widetilde\Delta \\
E' \\ E'' \end{bmatrix},
      \label{U-R-charfunc-full-density}
  \end{align}
where
\begin{align}\label{Ucharfunc'11-sh}
\Sigma^{(11)}_{n,m}&=\frac{1}{2}\left(\frac{I+
\begin{bmatrix} 0 & {\check b}_{n+m} \\ \check\omega_{n+m} & 0\end{bmatrix}}
{I-
\begin{bmatrix} 0 & {\check b}_{n+m} \\ \check\omega_{n+m} & 0\end{bmatrix}}+
\frac{I+
\begin{bmatrix} 0 & \overline{\check\omega}_{n+m} \\ \overline {\check b}_{n+m} & 0\end{bmatrix}}
{I-
\begin{bmatrix} 0 & \overline{\check\omega}_{n+m} \\ \overline {\check b}_{n+m} & 0\end{bmatrix}}
 \right):\begin{bmatrix} \Delta \\ \widetilde\Delta \end{bmatrix}\to
 \begin{bmatrix} \Delta \\ \widetilde\Delta \end{bmatrix}\\
 \nonumber\\
&=\begin{bmatrix} \frac{1}{2}\frac{1+{\check
b}_{n+m}\check\omega_{n+m}}{1-{\check b}_{n+m}\check\omega_{n+m}}&
\frac{{\check b}_{n+m}}{1-{\check b}_{n+m}\check\omega_{n+m}}\\ \\
\frac{\check\omega_{n+m}}{1-\check\omega_{n+m}{\check b}_{n+m}}&
\frac{1}{2}\frac{1+\check\omega_{n+m}{\check
b}_{n+m}}{1-\check\omega_{n+m}{\check b}_{n+m}}
\end{bmatrix} +
\begin{bmatrix}\frac{1}{2}\frac{1+\overline{\check\omega}_{n+m}\overline {\check b}_{n+m}}
{1-\overline{\check\omega}_{n+m}\overline {\check b}_{n+m}} &
\frac{\overline{\check\omega}_{n+m}}{1-\overline {\check b}_{n+m}\overline{\check\omega}_{n+m}}\\ \\
\frac{\overline {\check
b}_{n+m}}{1-\overline{\check\omega}_{n+m}\overline {\check
b}_{n+m}}& \frac{1}{2}\frac{1+\overline {\check
b}_{n+m}\overline{\check\omega}_{n+m}}
{1-\overline {\check b}_{n+m}\overline{\check\omega}_{n+m}}\end{bmatrix},\nonumber \\
\nonumber\\
\label{Ucharfunc'12-sh} \Sigma^{(12)}_{n,m}&=
\begin{bmatrix} 1 & \overline{\check\omega}_{n+m} \\
\check\omega_{n+m} & 1
\end{bmatrix}
\begin{bmatrix}
\frac{{\check s}_{1,n,m}}{1-{\check b}_{n+m}\check\omega_{n+m}}  &
0
\\ \\
0  & \frac{ \overline {\check s}_{2,n,m}}{1-\overline {\check
b}_{n+m}\overline{\check\omega}_{n+m}}
\end{bmatrix}\quad\quad\quad :\begin{bmatrix}
E' \\ E'' \end{bmatrix}\to\begin{bmatrix} \Delta \\
\widetilde\Delta \end{bmatrix},
\\
\label{Ucharfunc'21-sh}
\Sigma^{(21)}_{n,m}&= \Sigma^{(12)*}_{n,m}\\
&=
\begin{bmatrix}
\frac{\overline {\check
s}_{1,n,m}}{1-\overline{\check\omega}_{n+m}\overline {\check
b}_{n+m}} & 0
\\ \\
0  & \frac{ {\check s}_{2,n,m}}{1-\check\omega_{n+m}{\check
b}_{n+m}}
\end{bmatrix}
\begin{bmatrix} 1 & \overline{\check\omega}_{n+m} \\
\check\omega_{n+m} & 1
\end{bmatrix}
\quad\quad\quad :\begin{bmatrix} \Delta \\
\widetilde\Delta \end{bmatrix}\to
\begin{bmatrix}
E' \\ E'' \end{bmatrix},\nonumber\\
\label{Ucharfunc'-22-sh}
\Sigma^{(22)}&=\begin{bmatrix} 1 & \overline R \\
R & 1
\end{bmatrix}
\quad\quad\quad\quad\quad\quad\quad
\quad\quad\quad\quad\quad\quad\quad \quad\quad\quad\ \
 :\begin{bmatrix} E' \\ E''
\end{bmatrix}\to
\begin{bmatrix}
E' \\ E'' \end{bmatrix}.
\end{align}
Since the scale $\begin{bmatrix} \Psi_{n,m}^{\prime }
     &  \Psi_{n,m}^{\prime \prime }\end{bmatrix}$
     is $*$-cyclic for $\mathcal U_{R}$, we have
\begin{equation}\label{AAK}
\Sigma^{(11)}_{n,m}= \Sigma^{(12)}_{n,m} \Sigma^{(22)^{-1}}
\Sigma^{(21)}_{n,m}\ ,
\end{equation}
and $L^{R}$ is mapped unitarily onto $L^{{\check\Sigma}_{n,m}}$
 by
     \begin{equation}\label{L-R-L-Sigma-nm}
      \begin{bmatrix} f' \\ f'' \end{bmatrix}\to
     \begin{bmatrix}
\Sigma^{(12)}_{n,m} \Sigma^{(22)^{-1}}
\begin{bmatrix} f' \\ f'' \end{bmatrix}
\\ \\
\begin{bmatrix} f' \\ f'' \end{bmatrix}
\end{bmatrix}.
     \end{equation}

It follows from general feedback loading arguments that
$$
\ {\rm Clos \ }
     {\check\Sigma}_{0,n,m}
      \begin{bmatrix} 0\\0\\ t^n H^2_+ \\  \overline t^{\ m}H^2_-
      \end{bmatrix}=
\ {\rm Clos \ }
     \begin{bmatrix}
      {\check s}_{1,n,m} & 0 \\
      0 &
      \overline {\check s}_{2,n,m} \\
      1 & \overline {\check s}_{0,n+m}\\
      {\check s}_{0,n+m} &
      1
\end{bmatrix}
      \begin{bmatrix} t^n H^2_+ \\ \\ \overline t^{\ m}H^2_-
      \end{bmatrix}
$$
is unitarily mapped onto $\check{\mathcal H}_{n,m}$ by means of
multiplication by
$$
\begin{bmatrix}
0 & \frac{\overline {\check
s}_{1,n,m}\overline{\check\omega}_{n+m}}{1-\overline
{\check b}_{n+m}\overline{\check\omega}_{n+m}} & 1 & 0 \\ \\
\frac{{\check s}_{2,n,m}\check\omega_{n+m}}{1-{\check
b}_{n+m}\check\omega_{n+m}} & 0 &  0 & 1
\end{bmatrix}.
$$
This fact also follows from the next formula since Fourier
coefficients of ${\check s}_{0,n,m}$ and $R$ from $-\infty$ to
$-(n+m+1)$ agree:
\begin{equation}\label{Sigma-0-Sigma}
\begin{bmatrix} 0&0&1&0\\ 0&0&0&1 \end{bmatrix}{\check\Sigma}_{n,m}=
\begin{bmatrix}
0 & \frac{\overline {\check
s}_{1,n,m}\overline{\check\omega}_{n+m}}{1-\overline
{\check b}_{n+m}\overline{\check\omega}_{n+m}} & 1 & 0 \\ \\
\frac{{\check s}_{2,n,m}\check\omega_{n+m}}{1-{\check
b}_{n+m}\check\omega_{n+m}} & 0 &  0 & 1
\end{bmatrix}{\check\Sigma}_{0,n,m}.
\end{equation}
In particular, this implies, in view of (\ref{K-nm-0}),
(\ref{tilde-K-nm-0}) and (\ref{Sigma-0-Sigma}), that
\begin{equation}\label{K-nm}
     {\check K}_{n,m}=
     \begin{bmatrix}
\frac{\overline {\check
s}_{1,n,m}}{1-\overline{\check\omega}_{n+m}\overline
{\check b}_{n+m}} \\ \\
\frac{{\check s}_{2,n,m}\check\omega_{n+m}}{1-{\check
b}_{n+m}\check\omega_{n+m}}
\end{bmatrix}
     ={\check\Sigma}_{n,m}^{(21)}
\begin{bmatrix} 1\\0 \end{bmatrix}
\end{equation}
and
\begin{equation}\label{tilde-K-nm}
\widetilde {\check K}_{n,m}=
\begin{bmatrix}
\frac{\overline {\check
s}_{1,n,m}\overline{\check\omega}_{n+m}}{1-\overline
{\check\omega}_{n+m}\overline
{\check b}_{n+m}} \\ \\
\frac{{\check s}_{2,n,m}}{1-{\check b}_{n+m}\check\omega_{n+m}}
\end{bmatrix}
\overline t ={\check\Sigma}_{n,m}^{(21)}
\begin{bmatrix} 0\\ \overline t\end{bmatrix}.
\end{equation}
At the same time $\mathcal H^{\check\omega_{n+m}}$ is unitarily
mapped onto $\mathcal H_{n,m}^\perp$ by the formula
\begin{equation}\label{couple-perp}
\mathcal {\check H}_{n,m}^\perp=
\begin{bmatrix}
\frac{\overline {\check
s}_{1,n,m}}{1-\overline{\check\omega}_{n+m}\overline {\check
b}_{n+m}} & 0
\\ \\
0  & \frac{ {\check s}_{2,n,m}}{1-\check\omega_{n+m}{\check
b}_{n+m}}
\end{bmatrix}
\mathcal H^{\check\omega_{n+m}}.
\end{equation}
We can also write (unitarily equivalent)
$L^{{\check\Sigma}_{n,m}}$ realization for $L^R$ using
(\ref{L-R-L-Sigma-nm}). In this realization $\mathcal H_{n,m}$
looks as
$$ \ {\rm Clos \ }
     \begin{bmatrix}
      \Sigma^{(12)}_{n,m}\\ \\ \Sigma^{(22)}
\end{bmatrix}
      \begin{bmatrix} t^n H^2_+ \\ \\ \overline t^{\ m}H^2_-
      \end{bmatrix}=\ {\rm Clos \ }{\check\Sigma}_{n,m}
      \begin{bmatrix} 0\\0\\ t^n H^2_+ \\ \overline t^{\ m}H^2_-
      \end{bmatrix}.
$$
Since the factor in (\ref{couple-perp}) is
${\check\Sigma}_{n,m}^{(21)}
\begin{bmatrix} 1 & \overline{\check\omega}_{n+m} \\
\check\omega_{n+m} & 1 \end{bmatrix}^{-1}$, then
$L^{{\check\Sigma}_{n,m}}$ realization of $\mathcal
H_{n,m}^{\perp}$ looks as
$$
\begin{bmatrix} \Sigma^{(11)}_{n,m} \\ \\ \Sigma^{(21)}_{n,m}\end{bmatrix}
\begin{bmatrix} 1 & \overline{\check\omega}_{n+m} \\
\check\omega_{n+m} & 1 \end{bmatrix}^{-1} \mathcal
H^{\check\omega_{n+m}}.
$$
In view of (\ref{K-nm}), (\ref{tilde-K-nm}) and (\ref{AAK})
\begin{equation}\label{Sigma-K-nm-K-nm-0}
{\check K}_{n,m}\to
\begin{bmatrix}{\check\Sigma}_{n,m}^{(11)}\\ \\
{\check\Sigma}_{n,m}^{(21)}\end{bmatrix}
\begin{bmatrix} 1 \\ 0 \end{bmatrix}=
{\check\Sigma}_{n,m}\begin{bmatrix} 1 \\ 0 \\ 0 \\ 0
\end{bmatrix},
\ 
\widetilde {\check K}_{n,m}\to
\begin{bmatrix}{\check\Sigma}_{n,m}^{(11)}\\ \\
{\check\Sigma}_{n,m}^{(21)}\end{bmatrix}
\begin{bmatrix} 0 \\ \overline t \end{bmatrix}=
{\check\Sigma}_{n,m}\begin{bmatrix} 0 \\ \overline t \\ 0 \\ 0
\end{bmatrix}.
\end{equation}
We can see from the later formulas that in this realization
${\check K}_{n,m}$ reproduces the $0$ Fourier coefficient of the
first entry in $L^{{\check\Sigma}_{n,m}}$ and $\widetilde {\check
K}_{n,m}$ reproduces the $-1$ Fourier coefficient of the second
entry in $L^{{\check\Sigma}_{n,m}}$.

Just in case we mention here that
$$
{\check\Sigma}_{n,m}=
\begin{bmatrix}
\frac{1}{1-{\check b}_{n+m}\check\omega_{n+m}} &
\frac{\overline{\check\omega}_{n+m}}{1-\overline
{\check b}_{n+m}\overline{\check\omega}_{n+m}} & 0 & 0 \\ \\
\frac{\check\omega_{n+m}}{1-{\check b}_{n+m}\check\omega_{n+m}} &
\frac{1}{1-\overline
{\check b}_{n+m}\overline{\check\omega}_{n+m}} & 0 & 0 \\ \\
0 & \frac{\overline {\check
s}_{1,n,m}\overline{\check\omega}_{n+m}}{1-\overline
{\check b}_{n+m}\overline{\check\omega}_{n+m}} & 1 & 0 \\ \\
\frac{{\check s}_{2,n,m}\check\omega_{n+m}}{1-{\check
b}_{n+m}\check\omega_{n+m}} & 0 &  0 & 1
\end{bmatrix}{\check\Sigma}_{0,n,m}.
$$

\bigskip

\section{ Verblunsky coefficients}

\bigskip

In this section we discuss recurrent relations between vectors
${\check K}_{n,m}, \widetilde {\check K}_{n,m}$. Since the both
pairs of vectors ${\check K}_{n,m}, \widetilde {\check K}_{n+1,m}$
and $\widetilde {\check K}_{n,m}, {\check K}_{n,m+1}$ form
orthonormal bases for $H_{n+1,m+1}\ominus H_{n,m}$ they are
related by a unitary matrix
\begin{equation}\label{verbl-01}
\begin{bmatrix} {\check K}_{n,m} & \widetilde {\check K}_{n+1,m}\end{bmatrix} = \begin{bmatrix}
\widetilde {\check K}_{n,m} &
{\check K}_{n,m+1}\end{bmatrix} \begin{bmatrix} {\check\alpha}_{n,m} & {\check\rho}_{n,m} \\
\widetilde{\check\rho}_{n,m} & \widetilde{\check\alpha}_{n,m}
\end{bmatrix}.
\end{equation}
Note first that since the transformation matrix in
(\ref{verbl-01}) is unitary then either both ${\check\rho}_{n,m}$
and $\widetilde{\check\rho}_{n,m}$ equal zero or both do not. If
both equal zero then ${\check K}_{n,m}$ and $\widetilde {\check
K}_{n,m}$ are proportional, meaning that $\mathcal H_{n+1,
m}=\mathcal H_{n, m+1}$, which is impossible. Thus,
${\check\rho}_{n,m}\ne 0$ and $\widetilde{\check\rho}_{n,m}\ne 0$.
Using formulas (\ref{K-nm}) and (\ref{tilde-K-nm}) for ${\check
K}_{n,m}$ and $\widetilde {\check K}_{n,m}$ we can write
(\ref{verbl-01}) as
\begin{eqnarray}\label{verbl-01-explicit} &&\begin{bmatrix}
\frac{t^n\overline {\check a}_{n+m}}{1-\overline
{\check\omega}_{n+m}\overline {\check b}_{n+m}} &
\frac{t^n\overline {\check
a}_{n+m+1}\overline{\check\omega}_{n+m+1}}{1-\overline
{\check\omega}_{n+m+1}\overline {\check b}_{n+m+1}}
\\ \\
\frac{\overline t^{\ m} {\check
a}_{n+m}\check\omega_{n+m}}{1-{\check b}_{n+m}\check\omega_{n+m}}
& \frac{\overline t^{\ m+1} {\check a}_{n+m+1}}{1-{\check
b}_{n+m+1}\check\omega_{n+m+1}}
\end{bmatrix}
= \\ \nonumber\\ \nonumber\\
&&\begin{bmatrix} \frac{t^{n-1}\overline {\check
a}_{n+m}\overline{\check\omega}_{n+m}}{1-\overline
{\check\omega}_{n+m}\overline {\check b}_{n+m}} & \frac{t^n
\overline {\check
a}_{n+m+1}}{1-\overline{\check\omega}_{n+m+1}\overline {\check
b}_{n+m+1}}
\\ \\
\frac{\overline t^{\ m+1}{\check a}_{n+m}}{1-{\check
b}_{n+m}\check\omega_{n+m}} & \frac{\overline t^{\ m+1}{\check
a}_{n+m+1}\check\omega_{n+m+1}}{1-{\check
b}_{n+m+1}\check\omega_{n+m+1}}
\end{bmatrix}
\begin{bmatrix} {\check\alpha}_{n,m} & {\check\rho}_{n,m} \\
\widetilde{\check\rho}_{n,m} & \widetilde{\check\alpha}_{n,m}
\end{bmatrix}. \nonumber
\end{eqnarray}
Compare $(1,1)$ entries on the left and on the right, divide them
by $t^n$ and take the "zero" Fourier coefficient. We get
$$
\overline {\check a}_{n+m}(0)=\overline {\check
a}_{n+m+1}(0)\widetilde{\check\rho}_{n,m}.
$$
Therefore, according to our normalization,
$$\widetilde{\check\rho}_{n,m}=
\frac{\overline {\check a}_{n+m}(0)}{\overline {\check
a}_{n+m+1}(0)}= \frac{{\check a}_{n+m}(0)}{{\check
a}_{n+m+1}(0)}>0 .
$$
Comparing $(1,2)$ entries the
same way we get
$$\overline {\check a}_{n+m+1}(0)\overline{\check\omega}_{n+m+1}(0)=
\overline {\check a}_{n+m+1}(0)\widetilde{\check\alpha}_{n,m} ,
$$
i.e,
$$\widetilde{\check\alpha}_{n,m}=\overline{\check\omega}_{n+m+1}(0).$$
Comparing $(2,1)$ entries, multiplying by $t^{m+1}$ and taking
"zero" Fourier coefficient we get
$$
0={\check a}_{n+m}(0){\check\alpha}_{n,m}+{\check
a}_{n+m+1}(0)\check\omega_{n+m+1}(0)\widetilde{\check\rho}_{n,m}.
$$
Substituting the above formula for $\widetilde{\check\rho}_{n,m}$
we obtain
$$
{\check\alpha}_{n,m}=-\check\omega_{n+m+1}(0).
$$
Comparing $(2,2)$ entries, multiplying by $t^{m+1}$ and taking
"zero" Fourier coefficient we get
$$
{\check a}_{n+m+1}(0)={\check a}_{n+m}(0){\check\rho}_{n,m}+
{\check
a}_{n+m+1}(0)\check\omega_{n+m+1}(0)\widetilde{\check\alpha}_{n,m}.
$$
Since $\check\omega_{n+m+1}(0)=\overline{\widetilde\alpha}_{n,m}$
and
$1-|\widetilde{\check\alpha}_{n,m}|^2=|\widetilde{\check\rho}_{n,m}|^2$
(because the transformation matrix is unitary), we have
$$
{\check a}_{n+m+1}(0)|\widetilde{\check\rho}_{n,m}|^2={\check
a}_{n+m}(0){\check\rho}_{n,m}.
$$
Using the above formula for $\widetilde{\check\rho}_{n,m}$ we get
$${\check\rho}_{n,m}=
\frac{{\check a}_{n+m}(0)}{{\check
a}_{n+m+1}(0)}=\widetilde{\check\rho}_{n,m}.
$$
Thus, (\ref{verbl-01}) takes on the form
\begin{equation}\label{verbl-02}
\begin{bmatrix} {\check K}_{n,m} & \widetilde {\check K}_{n+1,m}\end{bmatrix} = \begin{bmatrix}
\widetilde {\check K}_{n,m} & {\check K}_{n,m+1}\end{bmatrix}
\begin{bmatrix}
{\check\alpha}_{n+m} & {\check\rho}_{n+m} \\
{\check\rho}_{n+m} & -\overline{\check\alpha}_{n+m}\end{bmatrix},
\end{equation}
where ${\check\rho}_{n+m}= \frac{{\check a}_{n+m}(0)}{{\check
a}_{n+m+1}(0)}>0$ and $
{\check\alpha}_{n+m}=-\check\omega_{n+m+1}(0)$ (because the
entries of the transformation matrix depend on the sum of the
indices only, we changed the notations). Note also that
${\check\rho}_{n+m}=\sqrt {1-|{\check\alpha}_{n+m}|^2}$ and
$|{\check\alpha}_{n+m}|<1$.

On the other hand the transformation matrix can be computed as
$$
\begin{bmatrix} \widetilde {\check K}_{n,m}^* \\
{\check K}_{n,m+1}^*\end{bmatrix}
\begin{bmatrix} {\check K}_{n,m} & \widetilde {\check K}_{n+1,m}\end{bmatrix} =
\begin{bmatrix} {\check\alpha}_{n+m} & {\check\rho}_{n+m} \\
{\check\rho}_{n+m} & -\overline{\check\alpha}_{n+m} \end{bmatrix}.
$$
We can compute $\langle {\check K}_{n,m}, \widetilde {\check
K}_{n,m}\rangle $ using $L^{{\check\Sigma}_{0,n,m}}$
representation (formulas (\ref{K-nm-0}) and (\ref{tilde-K-nm-0}))
or $L^{{\check\Sigma}_{n,m}}$ realization (formulas
(\ref{Sigma-K-nm-K-nm-0}))
\begin{eqnarray}
{\check\alpha}_{n+m}&=&\langle {\check K}_{n,m}, \widetilde
{\check K}_{n,m}\rangle =\langle {\check\Sigma}_{0,n,m}
\begin{bmatrix} 1\\0\\0\\0 \end{bmatrix},
\begin{bmatrix} 0\\ \overline t \\0\\0 \end{bmatrix}\rangle\nonumber\\
&=& ((2,1) {\rm\ entry\ of\ }{\check\Sigma}_{0,n,m})_{-1}=
 \frac{\overline {\check b}_{n+m}}{\overline t}(0).\nonumber
 \end{eqnarray}
 In particular we get
$$
\check\omega_{n+m+1}(0)=-\frac{\overline {\check
b}_{n+m}}{\overline t}(0).
$$
Note that since ${\check\rho}_j\le 1$, the sequence ${\check
a}_j(0)$ is increasing. On the other hand, since ${\check a}_j$ is
a Schur class analytic function, ${\check a}_j(0)\le 1$.
Therefore, the limit of ${\check a}_j(0)$ exists and
$$
0<\lim_{j\to\infty}{\check a}_j(0)\le 1.
$$
In particular, this implies that
$$
\lim_{j\to\infty}{\check\rho}_j=1 \quad {\rm and}\
\lim_{j\to\infty}{\check\alpha}_j=0.
$$

\section{A basis for $L^R$ and the matrix of $U_R$ in this basis}

\bigskip
We consider a chain of subspaces in $L^R$
\begin{equation}
\ldots \supset \mathcal H_{n-1,n-1}\supset \mathcal
H_{n,n-1}\supset \mathcal H_{n,n} \supset \mathcal
H_{n+1,n}\supset \mathcal H_{n+1,n+1}\supset \mathcal
H_{n+2,n+1}\supset\ldots \label{chain-of-subspaces-01}
\end{equation}
The first index increases first and then the second one increases.
Every subsequent subspace is of codimension one in the preceding
one. Since the union of subspaces (\ref{chain-of-subspaces-01}) is
dense in $L^R$, the sequence of vectors
\begin{equation}
\ldots {\check K}_{n-1,n-1}, \widetilde {\check K}_{n,n-1},
{\check K}_{n,n} , \widetilde {\check K}_{n+1,n}, {\check
K}_{n+1,n+1}, \widetilde {\check K}_{n+2,n+1}\ldots
\label{basis-01}
\end{equation}
forms an orthonormal basis for $L^R$.

We want to get the matrix of the operator $U_R$ in the basis
(\ref{basis-01}). Since $U_R$ is a linear operator we have from
(\ref{verbl-02})
\begin{equation}\label{U-R-Verblunsky1}
\begin{bmatrix} U_R {\check K}_{n,n} & U_R \widetilde {\check K}_{n+1,n}\end{bmatrix} = \begin{bmatrix}
U_R \widetilde {\check K}_{n,n} & U_R {\check
K}_{n,n+1}\end{bmatrix}
\begin{bmatrix}
{\check\alpha}_{2n} & {\check\rho}_{2n} \\
{\check\rho}_{2n} & -\overline{\check\alpha}_{2n}\end{bmatrix}.
\end{equation}
Since
$$
\widetilde {\check K}_{n,n}\in H_{n,n}\ominus H_{n,n+1},\quad
{\check K}_{n,n+1}\in H_{n,n+1}\ominus H_{n+1,n+1}
$$
then
$$
U_R\widetilde {\check K}_{n,n}\in H_{n+1,n-1}\ominus
H_{n+1,n},\quad U_R {\check K}_{n,n+1}\in H_{n+1,n}\ominus
H_{n+2,n}.
$$
Due to normalization ${\check a}_k(0)>0$ we have
$$
U_R\widetilde {\check K}_{n,n}=\widetilde {\check
K}_{n+1,n-1},\quad U_R {\check K}_{n,n+1}={\check K}_{n+1,n}.
$$
Applying (\ref{verbl-02}) with $n:=n, m:=n-1$ we obtain
\begin{equation}\label{Verblunsky2-(2n-1)}
\widetilde {\check K}_{n+1,n-1} =
\begin{bmatrix} \widetilde {\check K}_{n,n-1} &
{\check K}_{n,n}\end{bmatrix} \begin{bmatrix} {\check\alpha}_{2n-1} & {\check\rho}_{2n-1} \\
{\check\rho}_{2n-1} & -\overline{\check\alpha}_{2n-1}
\end{bmatrix}
\begin{bmatrix} 0 \\ 1 \end{bmatrix}.
\end{equation}
Applying (\ref{verbl-02}) with $n:=n+1, m:=n$ we obtain
\begin{equation}\label{Verblunsky2-(2n+1)}
{\check K}_{n+1,n} =
\begin{bmatrix} \widetilde {\check K}_{n+1,n} &
{\check K}_{n+1,n+1}\end{bmatrix} \begin{bmatrix} {\check\alpha}_{2n+1} & {\check\rho}_{2n+1} \\
{\check\rho}_{2n+1} & -\overline{\check\alpha}_{2n+1}
\end{bmatrix}
\begin{bmatrix} 1 \\ 0 \end{bmatrix}.
\end{equation}
Substituting (\ref{Verblunsky2-(2n-1)}) and
(\ref{Verblunsky2-(2n+1)}) to (\ref{U-R-Verblunsky1}), we get
\begin{eqnarray}\label{U-R-matrix}
&&\begin{bmatrix} U_R {\check K}_{n,n} & U_R \widetilde {\check K}_{n+1,n}\end{bmatrix} = \\
&&\begin{bmatrix} \widetilde {\check K}_{n,n-1} & {\check K}_{n,n}
& \widetilde {\check K}_{n+1,n} & {\check
K}_{n+1,n+1}\end{bmatrix}
\begin{bmatrix} {\check\rho}_{2n-1} & 0\\
-\overline{\check\alpha}_{2n-1} & 0 \\
0 & {\check\alpha}_{2n+1}\\
0 & {\check\rho}_{2n+1}
\end{bmatrix}
\begin{bmatrix} {\check\alpha}_{2n} & {\check\rho}_{2n} \\
{\check\rho}_{2n} & -\overline{\check\alpha}_{2n}
\end{bmatrix}.\nonumber
\end{eqnarray}
From this we see that the matrix of $U_R$ is a CMV matrix.

\bigskip

\section{Verblunsky coefficients as Schur parameters}

\bigskip

We summarize first results of the previous sections.
\begin{eqnarray}\label{U-R-matrix-rep}
&&\begin{bmatrix} t {\check K}_{n,n} & t \widetilde {\check K}_{n+1,n}\end{bmatrix} = \\
&&\begin{bmatrix} \widetilde {\check K}_{n,n-1} & {\check K}_{n,n}
& \widetilde {\check K}_{n+1,n} & {\check
K}_{n+1,n+1}\end{bmatrix}
\begin{bmatrix} {\check\rho}_{2n-1} & 0\\
-\overline{\check\alpha}_{2n-1} & 0 \\
0 & {\check\alpha}_{2n+1}\\
0 & {\check\rho}_{2n+1}
\end{bmatrix}
\begin{bmatrix} {\check\alpha}_{2n} & {\check\rho}_{2n} \\
{\check\rho}_{2n} & -\overline{\check\alpha}_{2n}
\end{bmatrix},\nonumber
\end{eqnarray}
where
\begin{equation}\label{K-nm-rep}
     {\check K}_{n,m}=
     \begin{bmatrix}
\frac{t^n \overline {\check a}_{n+m}}{1-\overline
{\check\omega}_{n+m}\overline
{\check b}_{n+m}} \\ \\
\frac{\bar t^m {\check a}_{n+m}\check\omega_{n+m}}{1-{\check
b}_{n+m}\check\omega_{n+m}}
\end{bmatrix}
     ={\check\Sigma}_{n,m}^{(21)}
\begin{bmatrix} 1\\0 \end{bmatrix},
\end{equation}
\begin{equation}\label{tilde-K-nm-rep}
\widetilde {\check K}_{n,m}=
\begin{bmatrix}
\frac{t^n\overline {\check
a}_{n+m}\overline{\check\omega}_{n+m}}{1-\overline
{\check\omega}_{n+m}\overline
{\check b}_{n+m}} \\ \\
\frac{\bar t^m {\check a}_{n+m}}{1-{\check
b}_{n+m}\check\omega_{n+m}}
\end{bmatrix}
\overline t ={\check\Sigma}_{n,m}^{(21)}
\begin{bmatrix} 0\\ \overline t\end{bmatrix},
\end{equation}
\begin{eqnarray}
{\check\alpha}_{n+m}=
 \frac{\overline {\check b}_{n+m}}{\overline t}(0)
 =-\check\omega_{n+m+1}(0),\nonumber
 \end{eqnarray}
\begin{eqnarray}
{\check\rho}_{n+m}=\sqrt{1-|{\check\alpha}_{n+m}|^2}. \nonumber
 \end{eqnarray}
Also ${\check\rho}_{n+m}=\frac{{\check a}_{n+m}(0)}{{\check
a}_{n+m+1}(0)}$. Substitute all this in (\ref{U-R-matrix-rep})
\begin{eqnarray}\label{U-R-matrix-subst}
&& \ \ \begin{bmatrix} \frac{t^{n+1} \overline {\check
a}_{2n}}{1-\overline {\check\omega}_{2n}\overline {\check s}_{2n}} &
\frac{t^{n+1}\overline {\check
a}_{2n+1}\overline{\check\omega}_{2n+1}}{1-\overline
{\check\omega}_{2n+1}\overline {\check s}_{2n+1}}
\\ \\
\frac{\bar t^{n-1} {\check a}_{2n}\check\omega_{2n}}{1-{\check
s}_{2n}\check\omega_{2n}} & \frac{\bar t^n {\check
a}_{2n+1}}{1-{\check s}_{2n+1}\check\omega_{2n+1}}
\end{bmatrix}
= \\
&&\begin{bmatrix} \frac{t^{n-1}\overline {\check
a}_{2n-1}\overline{\check\omega}_{2n-1}}{1-\overline
{\check\omega}_{2n-1}\overline {\check s}_{2n-1}} &
 \frac{t^{n} \overline
{\check a}_{2n}}{1-\overline {\check\omega}_{2n}\overline {\check
s}_{2n}} & \frac{t^{n}\overline {\check
a}_{2n+1}\overline{\check\omega}_{2n+1}}{1-\overline
{\check\omega}_{2n+1}\overline {\check s}_{2n+1}} & \frac{t^{n+1}
\overline {\check a}_{2n+2}}{1-\overline
{\check\omega}_{2n+2}\overline {\check s}_{2n+2}}
\\ \\
\frac{\bar t^{n} {\check a}_{2n-1}}{1-{\check
s}_{2n-1}\check\omega_{2n-1}}
 &
\frac{\bar t^{n} {\check a}_{2n}\check\omega_{2n}}{1-{\check
s}_{2n}\check\omega_{2n}}
 &
\frac{\bar t^{n+1} {\check a}_{2n+1}}{1-{\check
s}_{2n+1}\check\omega_{2n+1}}
 &
\frac{\bar t^{n+1} {\check a}_{2n+2}\check\omega_{2n+2}}{1-{\check
s}_{2n+2}\check\omega_{2n+2}}
\end{bmatrix}
\begin{bmatrix} {\check\rho}_{2n-1} & 0\\
-\overline{\check\alpha}_{2n-1} & 0 \\
0 & {\check\alpha}_{2n+1}\\
0 & {\check\rho}_{2n+1}
\end{bmatrix}
\begin{bmatrix} {\check\alpha}_{2n} & {\check\rho}_{2n} \\
{\check\rho}_{2n} & -\overline{\check\alpha}_{2n}
\end{bmatrix}.\nonumber
\end{eqnarray}
Or returning to (\ref{verbl-02})
\begin{equation}\label{verbl-02-rep}
\begin{bmatrix} {\check K}_{n,m} & \widetilde {\check K}_{n+1,m}\end{bmatrix} = \begin{bmatrix}
\widetilde {\check K}_{n,m} & {\check K}_{n,m+1}\end{bmatrix}
\begin{bmatrix}
{\check\alpha}_{n+m} & {\check\rho}_{n+m} \\
{\check\rho}_{n+m} & -\overline{\check\alpha}_{n+m}\end{bmatrix}.
\end{equation}
This can be rearranged as follows
\begin{equation}\label{verbl-02-rearange}
{\check\rho}_{n+m}\begin{bmatrix} {\check K}_{n,m+1} & \widetilde
{\check K}_{n+1,m}\end{bmatrix} =
\begin{bmatrix} {\check K}_{n,m} & \widetilde {\check K}_{n,m}\end{bmatrix}
\begin{bmatrix}
1& -\overline {\check\alpha}_{n+m} \\
-{\check\alpha}_{n+m} & 1 \end{bmatrix}.
\end{equation}
Substituting (\ref{K-nm-rep}) and (\ref{tilde-K-nm-rep}) in
(\ref{verbl-02-rearange}) we get
\begin{eqnarray}\label{verbl-02-subst}
&& {\check\rho}_{n+m}\begin{bmatrix} \frac{t^{n} \overline {\check
a}_{n+m+1}}{1-\overline {\check\omega}_{n+m+1}\overline {\check
b}_{n+m+1}} & \frac{t^{n}\overline {\check
a}_{n+m+1}\overline{\check\omega}_{n+m+1}}{1-\overline
{\check\omega}_{n+m+1}\overline {\check b}_{n+m+1}}
\\ \\
\frac{\bar t^{m+1} {\check
a}_{n+m+1}\check\omega_{n+m+1}}{1-{\check
b}_{n+m+1}\check\omega_{n+m+1}} & \frac{\bar t^{m+1} {\check
a}_{m+n+1}}{1-{\check s}_{m+n+1}\check\omega_{m+n+1}}
\end{bmatrix}
= \\ \nonumber \\
&&\begin{bmatrix} \frac{t^{n} \overline {\check
a}_{m+n}}{1-\overline {\check\omega}_{m+n}\overline {\check
s}_{m+n}} & \frac{t^{n-1}\overline {\check
a}_{n+m}\overline{\check\omega}_{n+m}}{1-\overline
{\check\omega}_{n+m}\overline {\check b}_{n+m}}
\\  \\
\frac{\bar t^{m}{\check a}_{m+n}\check\omega_{m+n}} {1-{\check
s}_{m+n}\check\omega_{m+n}} & \frac{\bar t^{m+1} {\check
a}_{n+m}}{1-{\check b}_{n+m}\check\omega_{n+m}}
\end{bmatrix}
\begin{bmatrix}
1& -\overline {\check\alpha}_{n+m} \\
-{\check\alpha}_{n+m} & 1 \end{bmatrix},\nonumber
\end{eqnarray}
or
\begin{equation}\label{verbl-02-rearange-through-Sigma}
{\check\rho}_{n+m}\Sigma^{(21)}_{n,m+1} =\Sigma^{(21)}_{n,m}
\begin{bmatrix} 1 & 0 \\ 0 & \overline t\end{bmatrix}
\begin{bmatrix}
1& -\overline {\check\alpha}_{n+m} \\
-{\check\alpha}_{n+m} & 1 \end{bmatrix}.
\end{equation}
We can write $\Sigma^{(21)}_{n,m}$ as
\begin{eqnarray}
\label{Sigma-21-renorm-1} \Sigma^{(21)}_{n,m} =\begin{bmatrix} t^n
& 0
\\ 0 & \overline t^m \end{bmatrix}
\Sigma^{(21)\prime}_{n+m},
\end{eqnarray}
where
\begin{eqnarray}
\label{Sigma-21-renorm-2}\Sigma^{(21)\prime}_{n+m}
=\begin{bmatrix}\overline A_{n+m} & 0
\\
0  & A_{n+m}
\end{bmatrix}
\begin{bmatrix} 1 & \overline{\check\omega}_{n+m} \\
\check\omega_{n+m} & 1
\end{bmatrix}
\end{eqnarray}
and
\begin{equation}\label{a-mn-renorm}
A_{n+m}=\frac{ {\check a}_{n+m}}{1-\check\omega_{n+m}{\check
b}_{n+m}}.
\end{equation}
Then (\ref{verbl-02-rearange-through-Sigma}) reads as
\begin{equation}\label{verbl-02-rearange-through-Sigma-2}
{\check\rho}_{n+m}
\begin{bmatrix} 1 & 0 \\ 0 & \overline t\end{bmatrix}
\Sigma^{(21)\prime}_{n+m+1} =\Sigma^{(21)\prime}_{n+m}
\begin{bmatrix} 1 & 0 \\ 0 & \overline t\end{bmatrix}
\begin{bmatrix}
1& -\overline {\check\alpha}_{n+m} \\
-{\check\alpha}_{n+m} & 1 \end{bmatrix}
\end{equation}
and (\ref{verbl-02-subst}) reads as
\begin{eqnarray}\label{verbl-02-subst-2}
&& {\check\rho}_{n+m}\begin{bmatrix} \overline A_{n+m+1} &
\overline A_{n+m+1}\overline{\check\omega}_{n+m+1}
\\ \\
 A_{n+m+1}\check\omega_{n+m+1} & A_{n+m+1}
\end{bmatrix}
= \\ \nonumber \\
&&\begin{bmatrix} \overline A_{n+m} & \overline
A_{n+m}\overline{\check\omega}_{n+m}\overline t
\\ \\
 A_{n+m}\check\omega_{n+m}t & A_{n+m}
\end{bmatrix}
\begin{bmatrix}
1& -\overline {\check\alpha}_{n+m} \\
-{\check\alpha}_{n+m} & 1 \end{bmatrix}.\nonumber
\end{eqnarray}
The second row of this relation reads as
\begin{eqnarray}\label{verbl-02-subst-3}
&& {\check\rho}_{n+m}\begin{bmatrix}
 A_{n+m+1}\check\omega_{n+m+1} & A_{n+m+1}
\end{bmatrix}
= \\ \nonumber \\
&&\begin{bmatrix}
 A_{n+m}\check\omega_{n+m}t & A_{n+m}
\end{bmatrix}
\begin{bmatrix}
1& -\overline {\check\alpha}_{n+m} \\
-{\check\alpha}_{n+m} & 1 \end{bmatrix},\nonumber
\end{eqnarray}
or, comparing entries,
$$
{\check\rho}_{n+m}
 A_{n+m+1}\check\omega_{n+m+1} = A_{n+m+1}
(\check\omega_{n+m}t - \check\alpha_{n+m})
$$
and
$$
{\check\rho}_{n+m}
 A_{n+m+1} = A_{n+m+1}
(1-t\check\omega_{n+m}t\overline{\check\alpha}_{n+m})
$$
In particular, this implies that
\eqbg{omega-recurrence}
\check\omega_{n+m+1}=\frac{t\check\omega_{n+m}-{\check\alpha}_{n+m}}{1-t\check\omega_{n+m}\overline{\check\alpha}_{n+m}}.
\eqen
%

%
%

\bigskip

\section{Asymptotic, Convergence, CMV basis and shift bases}

\bigskip

Space $L^R$ has an additional structure which is responsible for
the absolute continuity of the $2\times 2$ scattering measure of
the operator $U_R$ and for the special form of its density
$\begin{bmatrix} 1 & \overline R \\ R & 1 \end{bmatrix}$. Consider
the following two subspaces of $L^R$:
$$\mathcal L_1={\rm Clos}\ \bigcup\limits_{n=-\infty}^\infty \bigcap\limits_{m=1}^\infty
\mathcal H_{n,m}=\ \bigcup\limits_{n=-\infty}^\infty \mathcal
H_{n,\infty},\ {\rm where}\quad \mathcal H_{n,\infty}=
\bigcap\limits_{m=1}^\infty \mathcal H_{n,m}
$$
and
$$\mathcal L_2={\rm Clos}\ \bigcup\limits_{m=-\infty}^\infty \bigcap\limits_{n=1}^\infty
\mathcal H_{n,m}=\ \bigcup\limits_{m=-\infty}^\infty \mathcal
H_{\infty ,m},\ {\rm where}\quad \mathcal H_{\infty ,m}=
\bigcap\limits_{n=1}^\infty \mathcal H_{n,m}.
$$
\begin{lemma}
$$
\mathcal H_{n, \infty}=\begin{bmatrix} 1 \\ R
\end{bmatrix} t^n H^2_+, \quad
\mathcal L_{1}=\begin{bmatrix} 1 \\ R
\end{bmatrix} L^2
$$
and
$$
\mathcal H_{\infty , m}=\begin{bmatrix} \overline R \\
1
\end{bmatrix} \overline t^{\ m} H^2_-,\quad
\mathcal L_{2}=\begin{bmatrix} \overline R \\
1
\end{bmatrix} L^2
$$
\end{lemma}
Then $\mathcal L_{1}$ and $\mathcal L_{2}$ reduce $U_R$ and it
acts as a bilateral shift on $\mathcal L_{1}$ and on $\mathcal
L_{2}$. The corresponding wandering subspaces are defined as
$$
\mathcal H_{n,\infty}\ominus \mathcal H_{n+1,\infty}=\{
\begin{bmatrix} 1 \\ R \end{bmatrix}t^n\}=\{e_n\},
\ {\rm and}\ \mathcal H_{\infty ,m}\ominus \mathcal H_{\infty
,m+1}= \{
\begin{bmatrix} \overline R \\ 1 \end{bmatrix}\overline t^{\ m+1}\}=\{d_m\}.
$$

We start with proving some asymptotic properties of functions
${\check K}_{n,m}$ and $\tilde {\check K}_{n,m}$.
\begin{lemma}\label{L:Conv-Proj}
Let $\mathcal G$ be a Hilbert space and $\mathcal G_j$ be its
closed subspaces such that
$$
\mathcal G \supset \mathcal G_1 \supset \mathcal G_2 \supset\cdots
$$
Let $\mathcal G_\infty = \bigcap\limits_{j=1}^\infty \mathcal
G_j$. Let $P_j$ be orthogonal projection on $\mathcal G_j$. Then
for every $g\in\mathcal G$
$$ P_j g \to P_\infty g, \quad j\to\infty .$$
\end{lemma}
\begin{theorem}\label{T:Conv-CMV-to-Scatt}
\begin{equation}\label{en=lim}
{\check K}_{n,m}\to \begin{bmatrix} 1 \\ R \end{bmatrix}
t^n=e_n,\quad m\to\infty ,
\end{equation}
\begin{equation}\label{dm=lim}
\widetilde {\check K}_{n,m}\to
\begin{bmatrix} \overline R \\ 1 \end{bmatrix} \overline t^{m+1}=d_m,\quad n\to\infty
\end{equation}
and
$${\check a}_j(0)\to 1, \quad j\to\infty .$$
\end{theorem}
\begin{proof}
Note first that $\begin{bmatrix} 1 \\ R \end{bmatrix} t^n$ belongs
to $\mathcal H_{n,m}$ and does not belong to $\mathcal H_{n+1,m}$.
Using explicit formula (\ref{K-nm}) we compute
$$
\langle {\check K}_{n,m}, \begin{bmatrix} 1 \\ R \end{bmatrix}
t^n\rangle = (\overline t^{\ n}{\check K}_{n,m})_1 (0)={\check
a}_{n+m}(0).
$$
Therefore,
$$
\begin{bmatrix} 1 \\ R \end{bmatrix} t^n - {\check a}_{n+m}(0){\check K}_{n,m}=
P_{\mathcal H_{n+1,m}} \begin{bmatrix} 1 \\ R \end{bmatrix} t^n .
$$
By Lemma \ref{L:Conv-Proj},
$$
P_{\mathcal H_{n+1,m}} \begin{bmatrix} 1 \\ R \end{bmatrix} t^n\to
P_{\mathcal H_{n+1,\infty}} \begin{bmatrix} 1 \\ R \end{bmatrix}
t^n = 0, \quad m\to\infty .
$$
Since
$$
\Vert\begin{bmatrix} 1 \\ R \end{bmatrix} t^n - {\check
a}_{n+m}(0){\check K}_{n,m}\Vert^2= 1-{\check a}_{n+m}(0)^2,
$$
we get that ${\check a}_{n+m}(0)\to 1$ as $m\to\infty $. Thus, we
proved that
$${\check a}_{j}(0)\to 1, \quad j\to\infty .$$
Using again formulas (\ref{K-nm}) and (\ref{tilde-K-nm}), we can
see that
$$
\Vert\begin{bmatrix} 1 \\ R \end{bmatrix} t^n - {\check
K}_{n,m}\Vert^2= 2-2{\check a}_{n+m}(0)\to 0,\quad m\to\infty
$$
and
$$
\Vert\begin{bmatrix} \overline R \\ 1 \end{bmatrix} \overline t^{\
m+1}- \widetilde {\check K}_{n,m}\Vert^2 = 2-2{\check
a}_{n+m}(0)\to 0 \quad n\to\infty
$$
\end{proof}
Asymptotic of ${\check K}_{n,n}$ and $\widetilde {\check
K}_{n+1,n}$ can be obtained as a corollary of Theorem
\ref{T:Conv-CMV-to-Scatt}.
\begin{corollary}
\end{corollary}
$$
\Vert\begin{bmatrix} 1 \\ R \end{bmatrix} t^n - {\check
K}_{n,n}\Vert^2= 2-2{\check a}_{2n}(0)\to 0,\quad n\to\infty ,
$$
and
$$
\Vert\begin{bmatrix} \overline R \\ 1 \end{bmatrix} \overline t^{\
n+1} - \widetilde {\check K}_{n+1,n}\Vert^2= 2-2{\check
a}_{2n+1}(0)\to 0,\quad n\to\infty .
$$
\begin{remark}\label{conv-a-0}
We observe that $\lim\limits_{j\to\infty}{\check a}_j(0)=1$ is
equivalent to
\begin{equation}\label{prod-conv}
{\check a}_n(0)=\prod\limits_{j=n}^\infty {\check\rho}_j .
\end{equation}
Thus, existence of the scattering function $R$ implies convergence
of the above product, which is equivalent to the convergence of
the series
\begin{equation}\label{sum-conv}
\sum\limits_{j=n}^\infty
(1-{\check\rho}_j^2)=\sum\limits_{j=n}^\infty |{\check\alpha}_j|^2
<\infty .
\end{equation}
\end{remark}
Using Theorem \ref{T:Conv-CMV-to-Scatt} we can 
get an
expansion of $e_n$ and $d_m$ in the CMV basis constructed above in
terms of the Verblunski coefficients. To this end we notice that
$$
{\check K}_{n,n+2j}= U_R^{-j}{\check K}_{n+j, n+j}
$$
and
$$
\widetilde {\check K}_{n+2j+1,n}= U_R^{j}\widetilde {\check
K}_{n+j+1, n+j}.
$$

\bigskip

\section{Direct scattering for CMV matrices}

\bigskip

Given Verblunsky coefficients and corresponding unitary CMV matrix
in the basis (\ref{basis-01}), we want to obtain $R$. We assume
that convergence (\ref{en=lim}) and (\ref{dm=lim}) holds (which is
equivalent to convergence of product (\ref{prod-conv}),
equivalently, of the sum (\ref{sum-conv})). Thus we obtain
wandering systems of vectors $e_n$ and $d_m$. The pair of vectors
$e_0$ and $d_0$ is $*$-cyclic (because the pair ${\check K}_{0,0}$
and $\widetilde {\check K}_{1,0}$ is $*$-cyclic ???). Then $R$ is
the Adamjan-Arov scattering function. Its harmonic continuation is
computed as
$$
R(z)=d_0^*\{(I-zU_R^*)^{-1}+(I-\bar z U_R)^{-1}-I\}e_0.
$$

\bigskip

\section{Spectral representation}

\bigskip
Formula (\ref{U-R-matrix}) implies, in particular, that pair of
vectors ${\check K}_{n,n}$ and $\widetilde {\check K}_{n+1,n}$ is
$*$-cyclic for the operator $U_R$. Indeed...Proof...Spectral
measure of $U_R$ is defined in a similar way but using the vectors
${\check K}_{n,n}$ and $\widetilde {\check K}_{n+1, n}$ (this pair
is $*$-cyclic). It is easier to compute the spectral measure with
respect to the vectors ${\check K}_{n,n}$ and $t\widetilde {\check
K}_{n,n}$. By formulas (\ref{K-nm}), (\ref{tilde-K-nm}) and
(\ref{AAK}) it is equal to ${\check\Sigma}_{n,n}^{(11)}$ which is
defined by formula (\ref{Ucharfunc'11-sh}). Formula
(\ref{verbl-02}) implies that
\begin{equation}\nonumber
\begin{bmatrix} {\check K}_{n,m} & \widetilde {\check K}_{n+1,m}\end{bmatrix} = \begin{bmatrix}
{\check K}_{n,m} & \widetilde  {\check K}_{n,m}\end{bmatrix}
\begin{bmatrix}
1 & -\frac{\overline{\check\alpha}_{n+m}}{{\check\rho}_{n+m}} \\ \\
0 & \frac{1}{{\check\rho}_{n+m}}\end{bmatrix}.
\end{equation}
Thus the spectral measure with respect to ${\check K}_{n,n}$ and
$\widetilde {\check K}_{n+1, n}$ is given by
$$
\begin{bmatrix}
1 & 0 \\ \\
-\frac{{\check\alpha}_{2n}}{{\check\rho}_{2n}} &
\frac{t}{{\check\rho}_{2n}}
\end{bmatrix}
{\check\Sigma}_{n,n}^{(11)}
\begin{bmatrix}
1 & -\frac{\overline{\check\alpha}_{2n}}{{\check\rho}_{2n}} \\ \\
0 & \frac{\overline t}{{\check\rho}_{2n}}
\end{bmatrix}.
$$
Note that $\log\det{\check\Sigma}_{n,n}\in L^1$ because
${\check\Sigma}_{n,n}$ is factorized if $1-|\check\omega_{n+m}|^2$
is factorized, which is the case since $1-|R|^2$ is factorized.

\bigskip
\noindent {\bf Acknowledgment}. Alexander Kheifets wishes to thank
Franz Peherstorfer, Peter Yuditskii and the Group for Dynamical
Systems and Approximation Theory, University of Linz, Austria, for
their kind hospitality during a two week stay in the summer of 2007,
where a part of this manuscript was written.

\end{document}